\documentclass[11pt]{amsart}

\usepackage{pb-diagram}

\usepackage{amsfonts}
\usepackage{amsmath}
\usepackage{amssymb}

\newcommand{\kod}{\textnormal{kod}}

\newcommand{\dbar}{\overline{\partial}}

\newcommand{\ddt}[1]{\frac{\partial #1}{\partial t}}

\newcommand{\ddbar}{\sqrt{-1}\partial\dbar}

\newtheorem{theorem}{Theorem}[section]
\newtheorem{proposition}{Proposition}[section]
\newtheorem{lemma}{Lemma}[section]

\newtheorem{corollary}{Corollary}[section]

\begin{document}

\title{Bounding scalar curvature for global solutions of the K\"ahler-Ricci flow}

\author{Jian Song}

\address{Department of Mathematics, Rutgers University, Piscataway, NJ 08854}

\email{jiansong@math.rutgers.edu}

\thanks{Research supported in
part by National Science Foundation grants DMS-0847524 and DMS-0804095. The first named author is also supported in part by a Sloan Foundation Fellowship.}

\author{Gang Tian}

\address{BICMR and SMS, Peking University, Beijing, 100871, China}
\address{Department of Mathematics,
Princeton University, Princeton NJ 08544}

\email{tian@math.princeton.edu}

\begin{abstract} We show that the scalar curvature is uniformly bounded for the normalized K\"ahler-Ricci flow on a K\"ahler manifold with semi-ample canonical bundle. In particular, the normalized K\"ahler-Ricci flow has long time existence if and only if the scalar curvature is uniformly bounded, for K\"ahler surfaces, projective manifolds of complex dimension three, and for projective manifolds of all dimensions if assuming the abundance conjecture.

\end{abstract}

\maketitle

\section{Introduction}

Let $(X, g_0)$ be a K\"ahler manifold of $\dim_{\mathbb{C}}X = n\geq 2$ and $g_0$ a smooth K\"ahler metric. We consider the unnormalized K\"ahler-Ricci flow
\begin{equation}\label{unkrl}
\ddt{g} = - Ric(g), ~~~~  g|_{t=0} = g_0.
\end{equation}
By definition, $X$ is a minimal model if its canonical line bundle $K_X$ is nef. It is well-known \cite{Ts2, TiZha}  that the flow has a global solution on $X\times [0, \infty)$ if and only if the canonical line bundle $K_X$ is nef or equivalently, $X$ is a minimal model.  The abundance conjecture predicts that if the canonical line bundle $K_X$ over a projective manifold $X$ is nef, then it must be semi-ample, i.e., a sufficiently large power of $K_X$ is globally generated or base point free. In particular, the abundance conjecture holds for  projective manifolds of complex dimension no bigger than three \cite{Ka, M1, M2}. The aim of this paper is to investigate the behavior of the scalar curvature along the K\"ahler-Ricci flow on a K\"ahler manifold with semi-ample canonical line bundle.

\begin{theorem} \label{main1} Let $X$ be a K\"ahler manifold with $K_X$ being semi-ample, and $g(t)$ be the smooth global solution of the normalized K\"ahler-Ricci flow
\begin{equation}\label{nkrf}
\ddt{g} = - Ric(g)-g, ~~~~  g|_{t=0} = g_0.
\end{equation}
 Then there exists $C>0$ depending on $X$ and $g_0$, such that for all $t\in[0, \infty)$,
\begin{equation}\label{scbnd1}
|R(t)|_{L^\infty(X)} \leq C,
\end{equation}
where $R(t)$ is the scalar curvature of $g(t)$.

\end{theorem}

If we assume the abundance conjecture, then Theorem 1.1 holds for all projective minimal models, and it implies that the long time existence of the normalized K\"ahler-Ricci flow over a projective manifold is equivalent to the scalar curvature being uniformly bounded. In the case when $X$ is a minimal model of general type, i.e., $K_X$ is big and semi-ample, the uniform scalar curvature bound is proved by Zhang \cite{Z2}  and so the normalized K\"ahler-Ricci flow converges to the unique singular K\"ahler-Einstein metric with uniformly bounded scalar curvature.

When $0< \kod(X) < n$, a minimal model $X$ with semi-ample $K_X$ admits a Calabi-Yau fibration over its canonical model $X_{can}$,  it is shown in \cite{ST1, ST2}  that the normalized K\"ahler-Ricci flow collapses nonsingular Calabi-Yau fibres and the flow converges weakly to a generalized K\"ahler-Einstein metric $g_{can}$ on its canonical model $X_{can}$. In particular, on a dense Zariski open set of  $X_{can}$, such a canonical metric satisfies the generalized Einstein equation
\begin{equation}
Ric(g_{can}) = - g_{can} + g_{WP},
\end{equation}
where $g_{WP}$ is the Weil-Petersson metric induced from the Calabi-Yau fibration from $X$ to $X_{can}$. Then Theorem \ref{main1} shows that in this case, the normalized K\"ahler-Ricci flow collapses the Calabi-yau fibration with uniformly bounded scalar curvature. In particular, it improves the result in \cite{ST1}  for bounding the scalar curvature for the K\"ahler-Ricci flow on minimal elliptic surfaces.

When $\kod(X)=0$, $X$ is a Calabi-Yau manifold and it is proved in \cite{C} that the unnormalized K\"ahler-Ricci flow converges to the unique Ricci-flat K\"ahler metric exponentially fast.

After rescaling time and space simultaneously, we have the following immediate corollary from Theorem \ref{main1} for bounding the scalar curvature along the unnormalized K\"ahler-ricci flow.

\begin{corollary} \label{main2}

 Let $X$ be a K\"ahler manifold with $K_X$ being semi-ample, and $g(t)$ be the smooth global solution of the unnormalized K\"ahler-Ricci flow (\ref{unkrl}). Then there exists $C>0$ depending on $X$ and the initial K\"ahler metric, such that for all $t\in[0, \infty)$,
\begin{equation}\label{scbnd2}
|R(t)|_{L^\infty(X)} \leq C (1+t)^{-1}.
\end{equation}

\end{corollary}

For $X$ of $\kod(X)>0$, (\ref{scbnd2}) is optimal in the sense that there exists $c>0$ such that for all $t\geq 0$, there exists $z_t\in X$ such that $|R(t, z_t)|  \geq c(1+t)^{-1}.$ This can be seen from the simple example $X=E\times C$ where $E$ is an elliptic curve and $C$ is a curve with  genus greater than $1$.

We also give a criterion for long time existence of the normalized K\"ahler-Ricci flow for K\"ahler surfaces and projective manifolds of dimension $3$.
\begin{corollary} \label{main3}

Let $X$ be a K\"ahler surface or a projective manifold of complex dimension $3$. Then the normalized K\"ahler-Ricci flow (\ref{nkrf}) on $X$ admits a global solution if and only if the scalar curvature is uniformly bounded in time.

\end{corollary}

In general, it is natural to ask if the following holds  for the maximal solution of the unnormalized K\"ahler-Ricci flow on $X\times [0, T)$, where $X$ is a K\"ahler manifold and $T>0$ is the maximal existence time.

\begin{enumerate}

\item If $T<\infty$, then there exists $C>0$ such that $$-C\leq R(t)\leq C(T-t)^{-1}. $$

\item If $T=\infty$, then there exists $C>0$ such that $$|R(t) | \leq C (1+t) ^{-1}. $$

\end{enumerate}
In \cite{Pe2, SeT}, the answer to the first question is affirmative due to Perelman for the K\"ahler-Ricci flow  on Fano manifolds with finite time extinction. In \cite{Z3}, it is shown that if the K\"ahler-Ricci flow develops finite time singularity, the scalar curvature blows up at most of rate $(T-t)^{-2}$ if $X$ is projective and if the initial K\"ahler class lies in $H^2(X, \mathbb{Q})$.  One can even ask if the above estimates hold for the Ricci curvature along the unnormalized K\"ahler-Ricci flow.


\section{Volume estimates and parabolic Schwarz lemma}

Let $X$ be an $n$-dimensional  K\"ahler manifold with $K_X$ being semi-ample. Therefore the canonical ring $R(X, K_X)$ is finitely generated, and so the pluricanonical system $|mK_X |$ for sufficiently large $m\in \mathbb{Z}^+$, induces a holomorphic map
\begin{equation}\label{plurican}
\pi: X\rightarrow X_{can} \subset \mathbb{P}^N,
\end{equation}
where $X_{can}$ is the canonical model of $X$. The Kodaira dimension of $X$ is defined to be %
\begin{equation}
\kod(X) = \dim X_{can}.
\end{equation}
We always have
$$ 0\leq \kod(X) \leq \dim X =n. $$
In particular,
\begin{enumerate}

\item If $\kod(X) =n$, $X$ is birationally equivalent to its canonical model $X_{can}$, and $X$ is called a minimal model of general type.

\item If $0<\kod(X)<n$, $X$ admits a Calabi-Yau fiberation $$\pi: X \rightarrow X_{can}$$ over $X_{can}$ and a general fibre is a smooth Calabi-Yau manifold of complex dimension $n-\kod(X)$.

\item If $\kod(X)=0$, $X_{can}$ is a point and $X$ is a Calabi-Yau manifold with $c_1(X)=0$.

\end{enumerate}

Now we will reduce the normalized K\"ahler-Ricci flow to a parabolic Monge-Amp\`ere equation. Let $\mathcal{O}_{\mathbb{P}^N}(1)$ be the hyperplane bundle over $\mathbb{P}^N$ in (\ref{plurican}) and $\omega_{FS}\in [\mathcal{O}_{\mathbb{P}^N}(1)]$ be a Fubini-Study metric on $\mathbb{P}^N$. Then there exists $m>0$ such that $$m K_X = \pi^* \mathcal{O}_{\mathbb{P}^N}(1).$$  We define
$$\chi = \frac{1}{m} \pi^* \omega_{FS} \in [K_X]$$ and $\chi$ is a smooth nonnegative closed $(1,1)$-form on $X$. There also exists a smooth volume form $\Omega$ on $X$ such that $$Ric (\Omega)= -\ddbar \log \Omega = \chi.$$

Let $\omega_0$ be the initial K\"ahler metric of the normalized K\"ahler-Ricci flow (\ref{nkrf}) on $X$. Then K\"ahler class evolving along the normalized K\"ahler-Ricci flow is given by $$[\omega(t)] = (1-e^{-t}) [K_X] + e^{-t} [\omega_0]$$ and so $[\omega(t)]$ is a K\"ahler class for all $t\in [0, \infty)$. Therefore the normalized K\"ahler-Ricci flow starting with $\omega_0$ on $X$ has a smooth global solution on $X \times [0, \infty)$. We define the reference metric $$\omega_t = (1-e^{-t}) \chi + e^{-t} \omega_0.$$
 Then the Kahler-Ricci flow is equivalent to the following Monge-Ampere flow.
\begin{equation}\label{nmaeqn}
\ddt{\varphi} = \log \frac{ e^{(n-\kappa)t} (\omega_t + \ddbar \varphi)^n } {\Omega} - \varphi,
\end{equation}
where $\omega_t = \chi + e^{-t} ( \omega_0 - \chi)$ and $\kappa = \kod(X)$. In particular, $\chi=0$ when $\kod(X)=0$.

First we prove the following uniform estimates.

\begin{lemma} \label{c0}

There exists $C>0$ such that on $X\times [0, \infty)$,
\begin{equation}
\ddt{\varphi} \leq C
\end{equation}
and
\begin{equation}\label{c00}
|\varphi|  \leq C.
\end{equation}

\end{lemma}

\begin{proof} It is straightforward to show that there exists $C_1 >0$ such that $$ C_1^{-1} e^{  -n t}  \Omega  \leq \omega_t ^n \leq C_1 e^{-(n-\kappa) t } \Omega. $$ Then by applying the maximum principle for $\varphi$ and $e^{-t} \varphi$, there exists $C_2 >0$ such that on $X\times [0, \infty)$,
$$ \varphi \leq C_2, ~~~~~~~~~~~~  | e^{-t} \varphi |\leq C_2. $$

Straightforward calculations show that
\begin{equation}
\left(\ddt{} - \Delta\right) \ddt{\varphi} =  -e^{-t} tr_\omega (\omega_0 -\chi) - \ddt{\varphi} + (n-\kappa) .
\end{equation}
The uniform upper bound for $\ddt{\varphi}$ follows by applying  the maximum principle to $\ddt{\varphi} -  e^{-t}\varphi$  since there exists $C_3>0$ such that
\begin{eqnarray*}
 && \left(\ddt{} - \Delta\right) (\ddt{\varphi} -  e^{-t} \varphi) \\
 &=&  - e^{-t}  tr_\omega( \omega_t +\omega_0 -\chi  )  - (1+e^{-t}) \ddt{\varphi}  +  e^{-t} \varphi   + n e^{-t} +(n-\kappa) \\
 &\leq&  -  \ddt{\varphi} + C_3 .
 \end{eqnarray*}
Then we immediately obtain the uniform upper bound for $\ddt{\varphi}$. 

Now we will prove the lower bound for $\varphi$.  Rewrite the the parabolic Monge-Amp\`ere equation as 
$$ (\omega_t + \ddbar \varphi)^n = e^{ -(n-\kappa) t + \varphi + \ddt\varphi}\Omega. $$ Since $\omega_t  = (1-e^{-t}) \chi + e^{-t} \omega_0$ and there exists $C_4>0$ such that $C_4^{-1} \leq e^{(n-\kappa)t}  [\omega_t]^n \leq C_4$, $\varphi + \ddt\varphi \leq C_4$, by the results of \cite{DP, EGZ2}, there exists $C_5>0$ such that for all $t\in [0, \infty)$, 
\begin{equation}\label{c0estim}
 \sup_{z\in X}  \varphi(z, t)  - \inf_{z\in X} \varphi(z, t) \leq C_5.
 \end{equation}
 
  On the other hand, 

$$\int_X e^{\varphi + \ddt\varphi} \Omega = e^{(n-\kappa)t} [\omega_t]^n\geq C_4^{-1},$$ and so there exists $C_6>0$ such that for all $t\in [0, \infty)$, $$\sup_{z\in X} \left( \varphi+ \ddt\varphi\right) \geq - C_6. $$
Hence 
$$ \sup_{z\in X} \varphi(z, t) \geq \sup_{z\in X } \left( \varphi(z, t) + \ddt\varphi(z, t) \right) - \sup_{z\in X} \ddt\varphi(z, t)\geq  - \sup_{z\in X} \ddt\varphi(z, t) - C_6 $$ and so $\sup_{z\in X} \varphi(z, t)$ is uniformly bounded from below for all $t\in [0, \infty)$. Then (\ref{c00}) is proved by by applying the estimate (\ref{c0estim}).

\end{proof}

We now shall  prove a uniform  bound for  $\ddt \varphi$.  

\begin{proposition} \label{volbdp} There exists $C>0$ such that ,
\begin{equation}
\left| \ddt{\varphi} \right|_{L^\infty(X \times [0, \infty))}  \leq C.
\end{equation}
\end{proposition}

\begin{proof}
It suffices to prove a uniform lower bound for $\ddt{\varphi}$ by Lemma \ref{c0}.   

First we consider the following family of Monge-Amp\`ere equations for $s\in [0, \infty)$, 
\begin{equation}\label{smallpsi}
 (\omega_s + \ddbar\psi_s)^n   =  e^{ \psi_s} e^{- (n-\kappa) s} \Omega,
\end{equation}
where $$\omega_s = \chi + e^{-s} (\omega_0 -\chi)   .$$ There exists a unique smooth solution $\psi_s$ for each $s\in [0, \infty)$ \cite{Y1, A}.
It is straightforward to show by the maximum principle, that $\psi_s$ is uniformly bounded above on $X$ for all $s\in [0, \infty)$, i.e., there exists $ C_1 >0$ such that $$\psi_s\leq C_1$$ for all $s\in[0, \infty)$. Also $$\int_X e^{\psi_s} \Omega = e^{(n-\kappa)s} [\hat\omega_s]^n $$ is uniformly bounded from below for all $s\in [0, \infty)$. Therefore there exists $C_2>0$ such that $$-C_2 \leq \sup_{X} \psi_s \leq C_2$$ for all $s\in [0, \infty)$.  

By the results due to \cite{DP, EGZ2}  for solutions of degenerate Monge-Amp\`ere equations, there exists $C_3>0$ such that $$|\psi_s|_{L^\infty(X)} \leq C_3$$  for all $s\in [0, \infty)$. 

Let $\rho=\rho(t)$ be a smooth decreasing function defined on $[0, 1]$ such that
\begin{equation}
\rho(t) = \left\{        \begin{array}{ll}
    1, &  t\in [0, \frac{1}{3}] \\
    0,  & t\in [\frac{2}{3}, 1] .\\
  \end{array}  \right.
\end{equation}
We now define a smooth function $\Phi(z, t)$ on $ X \times [0, \infty)$ as follows.
\begin{equation}\label{bigphi}
\Phi( z, t) = \rho(t-m)  \psi_{m+1} (z)+       (1- \rho(t-m)) \psi_{m+2} (z), ~~~~(z, t)\in X \times [m, m+1],
\end{equation}
where $m$ is any nonnegative integer.
Then $\Phi$ is smooth in $X\times [0, \infty) $ and $$ -C_3 \leq \Phi \leq C_3.$$

Consider the same family of Monge-Amp\`ere equations as in equation (\ref{smallpsi})  for all $s\in [1, \infty)$ and define $\Phi$ as in equation (\ref{bigphi}).  Now we consider
$$ H = \ddt\varphi +  2 \varphi - \Phi.$$
The evolution of $H$ is given by
$$ \left(\ddt{} - \Delta\right) H = 2 \ddt{\varphi} +  tr_\omega \left( \chi +  \omega_t  + \ddbar \Phi\right) - (n+\kappa) - \ddt{\Phi}, $$
and there exists $C_4>0$ such that,
$$ \left(\ddt{} - \Delta\right) H \geq 2 \log \frac{ e^{(n-\kappa)t} \omega^n }{\Omega} + tr_{\omega} ( \chi + \omega_t + \ddbar \Phi) - C_4.$$
It is straightforward to check that for all $t\in [m, m+1]$, 
$$\chi+ \omega_t \geq  \omega_{m+1}, ~~~\chi + \omega_t \geq  \omega_{m+2}.$$

Suppose $t\in [m, m+1]$ for some nonnegative integer $m$.  Then  either $\rho(t-m) \geq 1/2$ or $(1-\rho(t-m) )\geq 1/2$ and we have
\begin{eqnarray*}
\chi + \omega_t+ \ddbar \Phi 
&\geq & \frac{1}{2} \min\left(   \chi+ \omega_t + \ddbar\psi_{m+1} , \chi  + \omega_t + \ddbar \psi_{m+2} \right) \\
&\geq & \frac{1}{2} \min\left(   \omega_{m+1} + \ddbar\psi_{m+1} , \omega_{m+2} + \ddbar \psi_{m+2} \right) 
\end{eqnarray*}
and so there exist $C_3, C_4 >0$ such that
\begin{eqnarray*}
(\chi + \omega_t + \ddbar \Phi )^n &\geq&  2^{-n} \min\left(  (\omega_{m+1} + \ddbar \psi_{m+1})^n   , (\omega_{m+2}+ \ddbar \psi_{m+2})^n       \right) \\
&\geq &  C_5   e^{- (n-\kappa) (m+2)}  \Omega  \\
&=&  C_5 e^{-(n-\kappa) (m+2 - t)}  e^{-(n-\kappa) t} \Omega\\
&\geq& C_6 e^{-(n-\kappa) t } \Omega.
\end{eqnarray*}
Suppose that $$H(z_0, t_0) = \inf_{X\times [0, T]} H(z, t)$$ with $t_0 \in [m, m+1]\cap[0, T]$ for some nonnegative integer $m$.  By combining the above estimates, we have at $(z_0, t_0)$, 
\begin{eqnarray*}
\left(\ddt{}- \Delta \right) H &\geq&  2 \log \frac{ e^{ (n-\kappa) t_0 } \omega^n } {\Omega} + C_7 \left(  \frac{ (\omega_{t_0} + \ddbar \Phi)^n }{\omega^n}    \right)^{1/n}  - C_8  \\
&\geq& 2 \log \frac{ e^{ (n-\kappa) t_0 } \omega^n } {\Omega} + C_7 \left(  \frac{ e^{-(n-\kappa)t _0} \Omega }{\omega^n}    \right)^{1/n}  - C_8  \\
&\geq& C_{9}  \left(  \frac{ e^{-(n-\kappa)t_0 } \Omega }{\omega^n}    \right)^{1/n}  - C_{10}\\
& \geq & C_{11} e^{- H/n} - C_{10}, 
\end{eqnarray*}
where the last inequality follows from the fact that $\varphi$ and $\Phi$ are uniformly bounded.  
By the maximum principle, $H(z_0, t_0) \geq n \log (C_{10}^{-1} C_{11}) $ and so for all $(z, t) $, $$ H(z, t) \geq \min \left( \min_X H(z,0),    n \log (C_{10}^{-1} C_{11} )\right), $$
and so it is uniformly bounded below. The proposition follows then immediately because $\varphi$ and $\Phi$ are uniformly bounded.

\end{proof}

The following calculation for the parabolic Schwarz lemma is given in \cite{ST1, ST2}.

\begin{lemma}\label{swaz} Let $\omega= \omega(t)$ be the solution of the normalized K\"ahler-Ricci flow (\ref{nkrf}).  If $\kod(X)>0$, then there exists $C>0$ such that on $X \times [0, \infty)$,
\begin{equation}
\left(\ddt{} - \Delta\right) tr_{\omega}(\chi) \leq tr_{\omega}(\chi) + C \left( tr_{\omega}(\chi)\right)^2 - |\nabla tr_{\omega}(\chi)|^2_g,
\end{equation}
where $\Delta$ is the Laplace operator associated to the evolving metric $g(t)$.
\end{lemma}

The following proposition is an improvement in \cite{ST1, ST2}. The fact that $\ddt\varphi$ is bounded below helps to get the Schwarz lemma and the lower bound of $\omega$ by $\chi$.

\begin{proposition} \label{chibound}  There exists $C>0$ such that on $X\times [0, \infty) $,
\begin{equation}
tr_\omega (\chi) \leq C.
\end{equation}

\end{proposition}

\begin{proof} Let $H = \log tr_{\omega} \chi - A\varphi$. Applying Lemma \ref{swaz},  there exist $C_1, C_2, C_3>0$  such that
$$ \left(\ddt{}- \Delta \right) H \leq - tr_\omega ( A \omega_t - C_1 \chi ) - C_2 \leq - tr_\omega (\chi ) - C_2 \leq - C_3e^{-H} - C_2.$$
By the maximum principle, $H$ is uniformly bounded and the proposition follows immediately.

\end{proof}


\section{Gradient estimates}

In this section, we will make use of the volume bound and parabolic Schwarz lemma to prove a parabolic analogue of Yau's gradient estimate \cite{LiYa, ChY},  to bound the scalar curvature. The parabolic gradient estimate is applied by Perelman to bound the scalar curvature for the K\"ahler-Ricci flow on Fano manifolds \cite{Pe2, SeT}.  This is also an improvement of  in \cite{ST1, ST2}.

We consider the normalized parabolic Monge-Amp\'ere flow (\ref{nmaeqn}) and let $u= \ddt{\varphi} + \varphi$. Since both $\ddt{\varphi}$ and $\varphi$ are uniformly bounded, there exists $A>0$ such that $$A- u \geq 1. $$

\begin{proposition} There exists $C>0$ such that
 \begin{equation}\label{gradest}  |\nabla u|_g^2 \leq C, \end{equation}
\begin{equation} \label{lapest}  - \Delta u \leq C . \end{equation}

\end{proposition}

\begin{proof} The proof is adapted from the calculations in \cite{ST1, ST2}. We assume that $\kod(X)>0$. When $\kod(X)=0$,  the proof of the proposition follows the same way since $\chi=0$ and it is in fact simpler.

First we note that
$$\left(\ddt{} - \Delta\right) u = tr_\omega(\chi) -\kappa. $$
The evolution for $|\nabla  u|_g^2$ and $\Delta u$ are given as below where $|\nabla u|_g^2= g^{i\bar j}u_i u_{\bar j}$ and $g=g(t)$ is the evolving metric associated to $\omega(t)$.
\begin{equation}
\left( \frac{\partial}{\partial t}-\Delta \right) |\nabla u|_g^2=|\nabla
u|_g^2+(\nabla
tr_{\omega}(\chi)\cdot\overline{\nabla}u+\overline{\nabla}
tr_{\omega}(\chi)\cdot\nabla u)-|\nabla\nabla
u|_g^2-|\overline{\nabla}\nabla u|_g^2,
\end{equation}
\begin{equation}
\left( \frac{\partial}{\partial t}-\Delta \right) \Delta u=\Delta
u+g^{i\overline{l}}g^{k\overline{j}}R_{k\overline{l}}u_{i\overline{j}}+\Delta
tr_{\omega}(\chi).
\end{equation}
Let $$ H = \frac{ |\nabla u|_g^2 }{A- u } + tr_\omega(\chi).$$ Then we have
\begin{eqnarray*}
&& \left( \ddt{} - \Delta \right) H\\
&=&  \left( \frac{|\nabla u|_g^2-|\nabla \nabla
u|_g^2-|\overline{\nabla}\nabla u|_g^2+(\nabla
tr_{\omega}(\chi)\cdot\overline{\nabla}u+\overline{\nabla}
tr_{\omega}(\chi)\cdot\nabla u)}{A-u} \right) \\
&&\\
&&  - \epsilon  \left( \frac{\nabla|\nabla
u|_g^2\cdot\overline{\nabla}u}{(A-u)^2}+\frac{\overline{\nabla}|\nabla
u|_g^2\cdot\nabla u}{(A-u)^2} \right) - 2 \epsilon  \frac{ | \nabla
u|_g^4}{(A-u)^3}- \frac{ 2(1-\epsilon) }{ A - u } Re \left( \nabla H \cdot \overline{\nabla} u  \right)  \\
&&\\
&& + \frac{2(1-\epsilon) }{ A- u } Re \left( \nabla u \cdot \overline{ \nabla } \left( tr_{\omega} (\chi)  \right) \right)   + ( tr_{\omega} (\chi) -\kappa ) \frac{ |\nabla u|_g^2 }{ (A- u)^2} + \left( \ddt{} -\Delta \right) tr_{\omega} ( \chi )
\end{eqnarray*}
Since $A-u$, $(A-u)^{-1}$, $tr_\omega(\chi)$ are uniformly bounded, by applying Lemma \ref{swaz}, Proposition \ref{volbdp}, Proposition \ref{chibound} and Schwarz inequality a few times, there exist $C_1, C_2>0$ depending on $\epsilon>0$ such that
\begin{eqnarray*}
&& \left( \ddt{} - \Delta \right) H \\
&\leq &  \left( \frac{|\nabla u|_g^2-|\nabla \nabla
u|_g^2-|\overline{\nabla}\nabla u|_g^2+(\nabla
tr_{\omega}(\chi)\cdot\overline{\nabla}u+\overline{\nabla}
tr_{\omega}(\chi)\cdot\nabla u)}{A-u} \right) \\
&&\\
&&  - \epsilon  \left( \frac{\nabla|\nabla
u|_g^2\cdot\overline{\nabla}u}{(A-u)^2}+\frac{\overline{\nabla}|\nabla
u|_g^2\cdot\nabla u}{(A-u)^2} \right) - 2 \epsilon  \frac{ | \nabla
u|_g^4}{(A-u)^3}- \frac{ 2(1-\epsilon) }{ A - u } Re \left( \nabla H \cdot \overline{\nabla} u  \right)  \\
&&\\
&& + \frac{2(1-\epsilon) }{ A- u } Re \left( \nabla u \cdot \overline{ \nabla } \left( tr_{\omega} (\chi)  \right) \right)   + ( tr_{\omega} (\chi) -\kappa ) \frac{ |\nabla u|_g^2 }{ (A- u)^2}  -  |\nabla tr_\omega(\chi)|^2 + C_1\\
&\leq& - C_2 \epsilon  |\nabla u|_g^4 - \frac{ 2- 2\epsilon}{A- u} Re (\nabla H\cdot \overline\nabla u) + C_3.
\end{eqnarray*}
By applying the maximum principle at $(z_0, t_0)$, for $H(z_0, t_0) = \max_{X\times [0, t]} H(t, z)$,  $H(z_0, t_0)$ is uniformly bounded and so is $H(z, t)$ on $X\times[0, \infty)$.  Then inequality (\ref{gradest}) follows immediately.

\medskip

We shall now prove inequality (\ref{lapest}).  Let $$K = - \frac{ \Delta u}{ A- u} + \frac{ 4 |\nabla u|_g^2}{A-u},  $$  then the evolution for $K$ is given by
\begin{eqnarray*}
&& \left( \ddt{} - \Delta \right) K \\
&&\\
&=&  \left(  \frac{ 4|\nabla u|_g^2-\Delta u - 4|\nabla \nabla u|_g^2 -3 | \nabla \overline{\nabla} u|_g^2 + g^{i\overline{j}} g^{k \overline{l}} \chi_{i\overline{j}} u_{ k \overline{ l}} - \Delta tr_{\omega} (\chi) + 8 Re \left( \nabla tr_{\omega} (\chi) \cdot \overline{\nabla} u \right)} { A - u } \right)  \\
&&\\
&&   + 4( tr_{\omega}(\chi)-\kappa) \frac{ |\nabla u|_g^2}{(A-u)^2}  - \frac{ 2 }{A-u} Re \left( \nabla K \cdot \overline{ u } \right) -  \left( tr_{\omega} (\chi) - \kappa \right) \frac{\Delta u}{ (A-u)^2}.
\end{eqnarray*}
\medskip
On the other hand, we have
$$ R_{i\bar j} = - u_{i\bar j} - \chi_{i\bar j}$$
and from Lemma \ref{swaz} and Proposition \ref{chibound}, there exists $C_4, C_5>0$ such that
\begin{eqnarray*}
&& - \Delta tr_{\omega}(\chi)\\
&=& \left(\ddt{} - \Delta\right) tr_\omega(\chi) -   \ddt{ } tr_\omega(\chi) \\
&=&  \left(\ddt{} - \Delta\right) tr_\omega(\chi) -  g^{i \overline{l}} g^{k\overline{j}} R_{k\overline{l}}  \chi_{ i \overline{j}} -   tr_\omega(\chi) \\
&=& \left(\ddt{} - \Delta\right) tr_\omega(\chi) + g^{i \overline{l}} g^{k\overline{j}} (u_{k\overline{l}} + \chi_{k \overline{l}} ) \chi_{ i \overline{ j }} -  tr_\omega(\chi)  \\
&\leq&  |\overline{\nabla} \nabla u|_g^2 - C_4| \nabla tr_\omega(\chi)|^2 + C_5 .
\end{eqnarray*}
Combining the above estimates with inequality (\ref{gradest}) and applying Schwarz inequality, there exists $C_6>0$ such that
\begin{eqnarray*}
 && \left(\ddt{}- \Delta \right) K  \\
 &\leq& - \frac{ |\nabla \overline\nabla u|_g^2}{ A- u} - \frac{ 2}{A- u} Re( \nabla K\cdot \overline \nabla u) + C_6\\
 &\leq& - \frac{(\Delta u)^2}{A-u}  - \frac{ 2}{A- u} Re( \nabla K\cdot \overline \nabla u)  + C_6.
 \end{eqnarray*}
 By applying the maximum principle at $(z_0, t_0)$ for $K(z_0, t_0) = \max_{X\times [0, t]} K(z, t)$,  $K(z_0, t_0)$ is uniformly bounded and so is $K (z, t)$ on $X\times[0, \infty)$.  Then inequality (\ref{lapest}) follows immediately.

\end{proof}


\section{Proof of the main theorems}

The scalar curvature $R(t)$ along the normalized K\"ahler-Ricci flow (\ref{nkrf}) can be expressed by
\begin{equation} \label{scalex}
R(t) = - \Delta u - tr_\omega(\chi).
\end{equation}

Now we can prove  Theorem \ref{main1}.

\begin{theorem} \label{main1eq} There exists $C>0$ such that on $X\times [0, \infty)$,
$$|R(t)| \leq C.$$

\end{theorem}

\begin{proof} It is a well-known fact by the maximum principle that the scalar curvature $R(t)$ is uniformly bounded from below. Hence it suffices to prove a uniform upper bound for $R(t)$, and it follows immediately from equation (\ref{scalex}) and  equation (\ref{lapest}).

\end{proof}

Corollary \ref{main2} follows from Theorem \ref{main1eq} immediately by rescaling time and space.

\bigskip

The abundance conjecture holds for projective manifolds of dimension no bigger than three and by the classification of complex surfaces, the canonical line bundle of a K\"ahler surface is always semi-ample if it is nef.  Therefore the canonical line bundle is semi-ample for K\"ahler surfaces and three dimensional projective manifolds of nef canonical line bundle.  Corollary \ref{main3} then follows from Theorem \ref{main1eq}.

\bigskip
\bigskip

\noindent{\bf Acknowledgements} The first named author would like to thank D.H. Phong, J. Sturm, V. Tosatti and members of the complex geometry and PDE seminar at Columbia University. He also thanks B. Guo and V. Datar for a number of useful conversations.

\end{document}